\documentclass[12pt,a4paper]{amsart}
\usepackage{amsfonts}
\usepackage{amsthm}
\usepackage{amsmath}
\usepackage{amscd}
\usepackage[latin2]{inputenc}
\usepackage{t1enc}
\usepackage[mathscr]{eucal}
\usepackage{indentfirst}
\usepackage{graphicx}
\usepackage{graphics}
\usepackage{pict2e}
\usepackage{epic}
\usepackage{amssymb}
\usepackage[colorlinks,citecolor=red,urlcolor=blue,bookmarks=false,hypertexnames=true]{hyperref} 
\numberwithin{equation}{section}
\usepackage[margin=3cm]{geometry}

\theoremstyle{plain}
\newtheorem{Th}{Theorem}[section]
\newtheorem{Lemma}[Th]{Lemma}

\newtheorem{Prop}[Th]{Proposition}
\newtheorem*{Theorem-non}{Theorem}
\newtheorem*{Theorem-non2}{Theorem}

 \theoremstyle{definition}
 \newtheorem*{Proof-non}{Proof of Theorem \ref{Maintheorem} assuming Propositions \ref{Prop1},\ref{Propm}}
\newtheorem*{Proof-non2}{Proof of (1)  ($\bf{m_{1}}$-estimate) in Proposition \ref{Propm} assuming Proposition \ref{Proposition 5.1}}
\newtheorem*{Proof-non3}{Proof of Theorem \ref{Maintheorem2} assuming Propositions \ref{Prop1},\ref{Propm}}
\newtheorem*{Proof-non4}{Proof of Proposition \ref{Prop1}}
\newtheorem*{Proof-non5}{Proof of Proposition \ref{Propm}}

\newtheorem{Def}[Th]{Definition}

\newtheorem{?}[Th]{Problem}

\begin{document}

\author{Jiseong Kim}
\address{The University of Mississippi, Department of Mathematics
Hume Hall 335
Oxford, MS 38677}
\email{Jkim51@olemiss.edu}
\title{
GOLDBACH'S PROBLEM
IN SHORT INTERVALS FOR NUMBERS WITH A MISSING DIGIT}

\begin{abstract} 
In this paper, by assuming a zero-free region for Dirichlet L-functions, we show that almost all even integers $n$ in a short interval $[x,x+x^{2/3+\varepsilon}]$ with a missing digit  are Goldbach numbers.
\end{abstract}

\maketitle
\section{Introduction} Let \( X \) be a power of a natural number \( g \), and let \([X, X+H]^{\ast}\) be the set of natural numbers \( X < n \leq X+H \) such that \( n \) has only digits from \(\{0, 1, 2, 3, \dots, g-1\} \setminus \{b\}\) in its base-\( g \) expansion, for some \( b \in \{2, 3, \dots, g-1\} \). The size of \([X, X+H]^{\ast}\) is approximately \( H^{\log (g - 1) / \log g} \) when \( X \) is a power of \( g \). 

Due to the lack of multiplicative structure among the elements in \([X, X+H]^{\ast}\), some standard approaches for summations of arithmetic functions over \([X, X+H]^{\ast}\) do not work. However, because \([X, X+H]^{\ast}\) has a nice Fourier transform, various interesting results have been proved (see \cite{Erdos1}, \cite{Erdos2}, \cite{Konyagin}, \cite{Banks}, \cite{Mau}). More recently, Maynard \cite{Jmay} proved that there are infinitely many primes with restricted digits, using various sieve methods and Fourier analysis. He established very strong upper bounds for 
\[
\sum_{n \in [1, X]^{\ast}} e(\alpha n)
\]
and its moments. This allowed him to apply the circle method to his problem, which is considered a binary problem. For other arithmetic functions over integers with restricted digits, see \cite{Nath}.

We say that an even number that can be written as a sum of two primes is a Goldbach number. In \cite{PP}, A. Perelli and J. Pintz showed that almost all even integers in \([X, X+X^{1/3+\varepsilon}]\) are Goldbach numbers. In this paper, we follow their argument to prove that almost all even numbers in \([X, X+H]^{\ast}\) are also Goldbach numbers under the following assumption:

\textbf{Assumption}: Let us denote $L(s,\chi)= \sum_{n=1}^{\infty} \chi(n)n^{-s}$ where $\chi$ is a Dirichlet character.  \( L(s,\chi) \neq 0 \) when \( \Re(s) > 1 - c_2 \), for some small fixed constant \( c_2 > 0 \).

In \cite{kimdiv}, we proved that  
\begin{equation}  
\sum_{n \in [X, X+H]^{\ast}} d_{2}(n) \ll \left|[X, X+H]^{\ast}\right| (\log X)^{3}  
\end{equation}  
when \( H \) is larger than \( X^{1/2} \). We will follow a similar argument to prove the following result, using recent results on exponential sum estimates in \cite{MSTT1}. 

\begin{Th}\label{Theorem1.1} Let $X^{3/5+\varepsilon} \ll H \ll X^{1-\varepsilon},$ and let $g$ be sufficiently large depends on $\varepsilon.$ Then 

$$\sum_{n \in [X,X+H]^{\ast}} d_{4}(n) \ll \ (\log X)^{7} | [x,x+H]^{\ast}|. $$
\end{Th} 
By applying Theorem 1.1 under the given assumption, we prove the following result:  
\begin{Th}\label{Theorem1.2}  
Let \( X^{2/3+\varepsilon} \ll H \ll X^{1-\varepsilon} \), and let \( g \) be sufficiently large depending on \( \varepsilon \). Assume the stated assumption holds. Then almost all \( 2n \) in \( [X, X+H]^{\ast} \) are Goldbach numbers.  
\end{Th}
For the long interval version, see \cite{cumberbatch}.

\subsection{Setting for Theorem 1.2}
 From now on, we assume that \( \varepsilon > 0 \) is a sufficiently small constant.
Let \( I_{1}=(X-H,X] \), \( I_{2}= (0,H] \), and let \( Q=\delta^{-1} \), where \( \delta= X^{-\varepsilon} \). Define  
\[
\begin{gathered}
R^*(2n)=R^*(2n, X, H)=\sum_{\substack{h+k=2n \\ h \in I_{1} \\ k \in I_{2}}} \Lambda(h) \Lambda(k), \\
M^*(2n)=M^*(2n, X, H)=\sum_{\substack{n+k=2n \\ n \in I_{1} \\ k \in I_{2}}} 1, \\
S_1(\alpha)=\sum_{n \in I_{1}} \Lambda(n) e(n \alpha), \quad S_2(\alpha)=\sum_{n \in I_{2}} \Lambda(n) e(n \alpha), \quad e(\alpha)=e^{2 \pi i \alpha},
\end{gathered}
\]
and let 
\[
\begin{aligned}
&T_1(\eta)=\sum_{n \in I_{1}} e(n \eta), \quad T_2(\eta)=\sum_{n \in I_{2}} e(n \eta), \\
&R_i(\eta, q, a)=S_i\left(\frac{a}{q}+\eta\right)-\frac{\mu(q)}{\phi(q)} T_i(\eta), \quad i=1,2, \\
&W_1(\chi, \eta)=\sum_{n \in I_{1}} \Lambda(n) \chi(n) e(n \eta)-\delta_\chi T_1(\eta), \\
&W_2(\chi, \eta)=\sum_{n \in I_{2}} \Lambda(n) \chi(n) e(n \eta)-\delta_\chi T_2(\eta), \quad \text{where } \delta_\chi=\begin{cases} 
1 & \text{if } \chi=\chi_0, \\ 
0 & \text{if } \chi \neq \chi_0,
\end{cases}
\end{aligned}
\]
\[
\sum_{a (q)}:=\sum_{a=1}^{q} \quad, \quad \sum_{a (q)}^{\ast}:=\sum_{\substack{a=1 \\ (a, q)=1}}^q.
\]

Also, \( \| \cdot \| \) will denote the distance to the nearest integer. Also we use the following notations:
\[
\begin{aligned}
c_q(m)&=\sum_{a(q)}^{\ast} e\left(\frac{m a}{q}\right), \quad \text{Ramanujan's sum}, \\
\tau(\chi)&=\sum_{a (q)}^{\ast}  \chi(a) e\left(\frac{a}{q}\right), \quad \text{Gauss's sum}, \\
\psi(x, \chi)&=\sum_{n \leqslant x} \Lambda(n) \chi(n), \\
N(\sigma, T, \chi)&=\left|\{\rho=\beta+i \gamma: L(\rho, \chi)=0, \beta \geqslant \sigma \text{ and }|\gamma| \leqslant T\}\right|, \\
N(\sigma, T, q)&= \sum_{\chi (\text{mod } q)} N(\sigma, T, \chi).
\end{aligned}
\]

\subsection{Lemmas}
Set $H=X^{2/3+\varepsilon}.$
By using the standard Hardy-Littlewood circle method,  we have 
$$
\sum_{k+m=2 n \atop k \in I_{1}, m \in I_{2}} \Lambda(k) \Lambda(m) =\int_{1 / Q}^{1+1 / Q} S_1(\alpha) S_2(\alpha) e(-2 n \alpha) d \alpha
$$

We divide up the interval $[1 / Q, 1+1 / Q]$ into Farey arcs of order $Q$, writing $I_{q, r}$ for the arc with centre at $r / q$. Thus

$$
I_{q, r} \subset\left[\frac{r}{q}-\frac{1}{q Q}, \frac{r}{q}+\frac{1}{q Q}\right]
$$

for $q \leq Q, 1 \leq r \leq q,(r, q)=1$. Let

$$
I_{q, r}^{\prime}=\left[\frac{r}{q}-\beta(\delta), \frac{r}{q}+\beta(\delta) \right], 
$$
where $D(\delta):=(\log X)^{-3^{6}-1}, \beta(\delta):=\frac{1}{H\delta D(\delta)}. $

The major and minor arcs are defined by

$$
\mathfrak{M}=\bigcup_{q \leq Q} \bigcup_{r=1 \atop (r,q)=1}^q I_{q, r}^{\prime}, \quad \mathfrak{m}=[1 / Q, 1+1 / Q] \backslash \mathfrak{M},
$$

respectively.

First, we need the follwing zero density estimates for the Dirichlet L-functions.

\begin{Lemma}\cite{LMC} For $1 \geq \sigma \geq 4/5, T \gg 1, $ we have 
\begin{equation}
N(\sigma, T, q) \ll_{\varepsilon}\left(qT\right)^{(2+\varepsilon)(1-\sigma)+o(1)}.
\end{equation}
\end{Lemma}

By following the standard argument, we can estimate upper bounds for \( W_{i}(\chi, \eta) \).

\begin{Lemma} Let $\chi$ be a Dirichlet character and let $T \gg 1.$ Then
\begin{equation}   
\sum_{n \in I_{1}} \Lambda(n) \chi(n) = \delta_{\chi} H - \sum_{\rho = \beta+ i\gamma \atop |\gamma|<T} \frac{X^{\rho} - (X-H)^{\rho}}{\rho} + O\left(X (\log qX)^{2} T^{-1}\right),  
\end{equation}  
where the summation runs over the nontrivial zeros \( \rho \) of \( L(s, \chi) \).  
\end{Lemma}

\begin{proof}
See \cite[Chapter 7]{Harman}.
\end{proof}
Note that 
$$
R_{i}(\eta, q, a)=\frac{1}{\phi(q)} \sum_\chi \chi(a) \tau(\bar{\chi}) W_{i}(\chi, \eta)+O(\sqrt{H}).$$
Therefore, by using the above lemmas, we can obtain the following bounds, which will be applied to some parts of the major arcs.
\begin{Lemma}\label{Lemma1.5} Assume the stated assumption holds. Then
\begin{equation}\label{1.5}\begin{split}
&\sum_{q \leq Q} \sum_{ a (q)}^{\ast} \frac{1}{\phi(q)} \int_{|\eta|<\beta(\delta)}\left(T_1(\eta) R_2(\eta, q, a)+T_2(\eta) R_1(\eta, q, a)\right) e(-2 n \eta) d \eta = o(H),
\\&\sum_{q \leq Q} \sum_{ a (q)}^{\ast}  \int_{|\eta|<\beta(\delta)} R_1(\eta, q, a) R_2(\eta, q, a) e(-2 n \eta) d \eta = o(H).
\end{split}\end{equation}
\end{Lemma}
\begin{proof}
By applying summation by parts, it is easy to see that when $|\eta|\leq \beta(\delta),$ 
$|W_{i}(\chi,\eta)|\ll H\beta(\delta) |W_{i}(\chi,0)|.$
Therefore, by using the bound $|\tau(\chi)|\leq q^{1/2},$ we have  $$\left|R_{i}(\eta,q,a)\right|\ll \frac{H\beta(\delta)}{q^{1/2}}\sum_{\chi} |W_{i}(\chi,0)| + O(\sqrt{H}).$$
Therefore,
$$\int_{|\eta|<\beta(\delta)} \left|R_{i}(\eta,q,a)\right|^{2} d\eta \ll \beta(\delta)\left(\frac{H^{2}\beta(\delta)^{2}}{q}\sum_{\chi} |W_{i}(\chi,0)|^{2} + H\right). $$
Note that $\beta(\delta)=\frac{(\log X)^{3^{6}+1}}{H\delta}.$
Therefore, the left-hand side of the first equation in \eqref{1.5} is bounded by 
\begin{equation}\begin{split}&\sum_{q \leq Q} \sum_{ a (q)}^{\ast} \frac{1}{\phi(q)q^{1/2}} H^{1/2} \beta(\delta)^{1/2}\left( H^{2}\beta(\delta)^{2}\sum_{\chi} |W_{i}(\chi,0)|^{2} + H\right)^{1/2}
\\& \quad \ll H^{3/2}\beta(\delta)^{3/2} \log Q \left(\sum_{q \leq Q} \left(\sum_{\chi} |W_{i}(\chi,0)|^{2} + \frac{1}{H\beta(\delta)^{2}}\right)\right)^{1/2}.
\end{split}\end{equation}
By following the argument in \cite[Chapter 7]{Harman}, setting $T= \frac{qX(\log X)^{2}}{X^{3/5}},$ with the assumption, we have 
\begin{equation}\begin{split} & \sum_{\chi} |W_{i}(\chi,0)|^{2}\ll H^{2}\sum_{\rho}X^{2\sigma-2} \left(\sum_{\chi (\textrm{mod }q)} N(\sigma,T,\chi) \right)
\\& \ll \left(\frac{H}{X}\right)^{2} \left(X^{2-2c_{2}}(qT)^{2c_{2}+o(1)}\right) 
\\& \ll H^{2}q^{5\varepsilon}X^{-1.25c_{2}+o(1)}.
\end{split}\end{equation}
Therefore, the left hand-side of $\eqref{1.5}$ is bounded by 
$$\delta^{-3/2} Q \log Q H X^{-0.625c_{2}+o(1)}.$$
Also, 
$$\int_{|\eta|<\beta(\delta)} \left|R_{i}(\eta,q,a)\right|^{2} d\eta \ll H^{4} \beta(\delta)^{3} q^{5\varepsilon-1}X^{-1.25c_{2}+o(1)} \ll H\delta^{-3} (\log X)^{-3( 3^{6}+1)} X^{-1.25c_{2}+o(1)}. $$
\end{proof}
Now, we show that under the assumption, the exponential sum with weight $\Lambda(n)$ is small when $\alpha \in m.$ Since $H\beta(\delta) \gg 1,$ we can use the following lemma.
\begin{Lemma}\label{Lemma1.6} Let the stated assumption hold, and let $X^{2/3+\varepsilon} \ll H \ll X^{1-\varepsilon}.$ When $\alpha \in m,$
$\max\left(S_{1}(\alpha), S_{2}(\alpha)\right) \ll HX^{-c_{3}\varepsilon}$ for some $c_{3}>0.$
\end{Lemma}
\begin{proof}
The proof comes from the argument in \cite{MS1}. The only differnce is, we use $\delta$ as a power of $X^{-\varepsilon}.$ One may use \cite[Lemma 3.4]{MRSTT2} instead of \cite[Lemma 2.3]{MS1} with the assumption. 

\end{proof}
The following lemma shows that, on average, \[ F_{[X,X+H]}(\alpha) := |[X,X+H]^{\ast}|^{-1} \left| \sum_{n \in [X,X+H]^{\ast}} e(n\alpha) \right|\] is of size  
\[
\left|[X,X+H]^{\ast}\right|^{-1+ \frac{\log \left((\log g)+1\right)}{\log (g-1)}}.
\]  

\begin{Lemma}\label{may}  
Let \( g \) be sufficiently large. Then  
\[
\int_{[0,1]} F_{[X,X+H]}(\alpha) \, d\alpha \ll \left|[X,X+H]^{\ast}\right|^{-1+ \frac{\log \left((\log g)+1\right)}{\log (g-1)}}.
\]  
\end{Lemma}  
To treat the main terms in the major arcs, we need the following lemma. 
\begin{Lemma}  
Let \( l \mid q \), and let \( (am', q) = 1 \). Then  
\[
\sum_{b (q)} \sum_{a(q)}^{\ast} e\left(\frac{a(m'b - n)}{q}\right) \ll d_{2}(q)(n, q).
\]  
\end{Lemma}  

\begin{proof}  
Note that  
\begin{equation}\label{ramanujan_formula}
c_{q}(n) = \mu\left(\frac{q}{(q, n)}\right) \frac{\phi(q)}{\phi\left(\frac{q}{(q, n)}\right)}.
\end{equation}  

By rearranging the double sum, it can be written as  
\[
\sum_{a(q)}^{\ast} e\left(-\frac{an}{q}\right) \sum_{l|b} c_{q/l}(am') = \sum_{a(q)}^{\ast} e\left(-\frac{an}{q}\right) \sum_{l|q} \mu(q/l).
\]  

By \eqref{ramanujan_formula}, the proof is complete.  
\end{proof}
\subsection{Proof of Theorem 1.3 assuming Theorem 1.2}
We basically follow the argument in \cite{PP}. The major arcs contribution is 
$$ 
\begin{aligned}\label{HARDY}
&\int_{\mathfrak{M}} S_1(\alpha) S_2(\alpha) e(-2 n \alpha) d \alpha \\&\quad =  \sum_{q<Q} \sum_{a (q)}^{\ast} e\left(-\frac{2 n a}{q}\right)\left\{\frac{\mu(q)^2}{\phi(q)^2} \int_{|\eta|<\beta(\delta)} T_1(\eta) T_2(\eta) e(-2 n \eta) d \eta\right. \\
& +\frac{\mu(q)}{\phi(q)} \int_{|\eta|<\beta(\delta)}\left(T_1(\eta) R_2(\eta, q, a)+T_2(\eta) R_1(\eta, q, a)\right) e(-2 n \eta) d \eta \\
& \left.+\int_{|\eta|<\beta(\delta)} R_1(\eta, q, a) R_2(\eta, q, a) e(-2 n \eta) d \eta\right\}.
\end{aligned}
$$
The first sum on the right hand-side of \eqref{HARDY} is 
$$ \sum_{q<Q} \sum_{a (q)}^{\ast} e\left(-\frac{2 n a}{q}\right) \frac{\mu(q)^2}{\phi(q)^2} \int_{[0,1]} T_1(\eta) T_2(\eta) e(-2 n \eta) d \eta + O\left(\log Q \int_{|\eta|>\beta(\delta)} \left(\frac{1}{|\eta|}\right)^{2} d\eta\right).$$
Since $\beta(\delta)^{-1} = H\delta (\log X)^{-3^{6}-1},$
the error term is $o(H).$
Also, we have 
\begin{equation}\begin{split}  \sum_{q<Q} \sum_{a (q)}^{\ast} e\left(-\frac{2 n a}{q}\right) \frac{\mu(q)^2}{\phi(q)^2} \int_{[0,1]} T_1(\eta) T_2(\eta) e(-2 n \eta) d \eta &= M^*(2 n) \sum_{q=1}^{\infty} \frac{\mu(q)^2}{\phi(q)^2} c_q(-2 n) \\& + O\left(H\left|\sum_{q>Q} \frac{\mu(q)^2}{\phi(q)^2} c_q(-2 n)\right|\right).\end{split}\end{equation}
Note that $d_{2}(n)^{2} \leq d_{4}(n).$ Therefore, by using Theorem \ref{Theorem1.1},
\begin{equation}\begin{split} \left(\sum_{n \in [X,X+H]^{\ast}} H^{2}\left|\sum_{q>Q} \frac{\mu(q)^2}{\phi(q)^2} c_q(-2 n)\right|^{2}\right)
&\ll 
H^2 \sum_{2 n \in [X,X+H]^{\ast}}\left\{\sum_{d \mid 2 n} \frac{1}{\phi(d)} \min \left(\frac{d}{Q}, 1\right)\right\}^2
\\&\ll H^{2} Q^{-2} \log Q \sum_{j \in [X,X+H]^{\ast}} d_{2}^2(j)
\\&\ll H^{2} Q^{-2} \log Q \sum_{j \in [X,X+H]^{\ast}} d_{4}(j)
\\&\ll o\left(H^{2}|[X.X+H]^{\ast}|\right).
\end{split}\end{equation}
By using Lemma \ref{Lemma1.5}, the second and third sums in \eqref{HARDY} are bounded by \( o(H) \). Since the main term  
\[
M^{\ast}(2n)\sum_{q=1}^{\infty} \frac{\mu(q)^2}{\phi(q)^2} c_q(-2n) \asymp H,
\]  
it follows that the contribution from the major arcs gives the asymptotic for almost all \( 2n \in [X, X+H]^{\ast} \).  

Now, let us consider the minor arcs. By Lemma \ref{Lemma1.6}, we have \( S_{i}(\alpha) \ll HX^{-c\varepsilon} \) for some \( c > 0 \). Therefore, combining the upper bound in Lemma \ref{may}, we have 
\[
\int_{m} \left|S_{1}(\alpha) S_{2}(\alpha)\right| \left|\sum_{n \in [X, X+H]^{\ast}} e(-n\alpha) \right|\, d\alpha \ll H^{2} X^{-2c\varepsilon} H^{\log((\log g)+1)/\log(g-1)}.
\]  
Since $|[X,X+H]^{\ast}| \asymp H^{\log g-1/\log g},$ the above upper bound is $o(H |[X,X+H]^{\ast}|). $

\section{Proof of Theorem 1.1}
In this section, we use the following notations:

\[
S_{4}(\alpha;H) := \sum_{n=X}^{X+H} d_{4}(n) e(n\alpha), \quad S_{4}(\alpha):= \sum_{n=1}^{X} d_{4}(n) e(n\alpha),
\]
\[
S_{4}^{\sharp}(\alpha;H) := \sum_{n=X}^{X+H} d_{4}^{\sharp}(n) e(n\alpha), \quad S_{4}^{\sharp}(\alpha):= \sum_{n=1}^{X} d_{4}^{\sharp}(n) e(n\alpha),
\]
\[ S_{[X,X+H]^{\ast}}(\alpha) := \sum_{n \in [X,X+H]^{\ast}} e(n\alpha)\].

For the approximant of \( d_{4}(n) \), we use the following in \cite{MSTT1}:
$$
d_4^{\sharp}(n):=\sum_{\substack{m \leq R_4^{6} \\ m \mid n}} P_m(\log n),
$$
where $R_4$ is the parameter

$$
R_4:=X^{\frac{\varepsilon}{40}},
$$

where the polynomials $P_m(t)$ (which have degree at most $3$ ) are given by the formula

$$
P_m(t):=\sum_{j=0}^{3}\binom{4}{j} \sum_{\substack{m=n_1 \cdots n_{3} \\ n_1, \ldots, n_j \leq R_4 \\ R_4<n_{j+1}, \ldots, n_{3} \leq R_4^2}} \frac{\left(t-\log \left(n_1 \cdots n_j R_4^{4-j}\right)\right)^{4-j-1}}{(4-j-1)!\log ^{4-j-1} R_4} .
$$
Note that $R_{4}^{6}= x^{3\varepsilon/20}.$
It is easy to see that 
$P_{m}(\log n )\ll_{\varepsilon} d_{3}(m)$ when $n \ll X.$ 
\subsection{Propositions} 
The following proposition shows that $d_{4}^{\sharp}(n)$ is a good approximation for $d_{4}(n)$.
\begin{Prop}\label{PROP1}\cite[Theorem 1.1 (iii)]{MSTT1} Let $X^{3/5+\varepsilon} \leq H \leq X^{1-\varepsilon}.$ Then
$$\left| \sum_{X \leq b \leq X + H} \left(d_{4}(n) - d_{4}^{\sharp}(n)\right) e(n\alpha)\right| \ll Hx^{-c\varepsilon}$$ for some $c>0.$
\end{Prop}
By using this, we can reduce our problem to the case of $d_{4}^{\sharp}(n).$
\begin{Prop} Let $X^{3/5+\varepsilon} \leq H \leq X^{1-\varepsilon}.$ Then 
\begin{equation}\begin{split} \sum_{ n \in [X,X+H]^{\ast}} \int_{[0,1]} S_{4}(\alpha;H) e(-n\alpha) d\alpha  &=\sum_{ n \in [x,x+H]^{\ast}} \int_{[0,1]} S_{4}^{\sharp}(\alpha;H) e(-n\alpha) d\alpha \\& \quad + o(|[X,X+H]^{\ast}|).\end{split}\end{equation}
\end{Prop}
\begin{proof} By using the trivial bound and Lemma \ref{may}, 
\begin{equation}\begin{split} &\sum_{ n \in [X,X+H]^{\ast}} \int_{[0,1]} \left(S_{4}(\alpha;H) - S_{4}^{\sharp}(\alpha;H)\right) e(-n\alpha) d\alpha \\& \ll sup_{\alpha \in [0,1]} \left|S_{4}(\alpha;H)- S_{4}^{\sharp}(\alpha;H)\right| \int_{[0,1]} \left|\sum_{n \in [X,X+H]^{\ast}} e(-n\alpha)\right|d\alpha 
\\& \ll HX^{-c\varepsilon} |[X,X+H]^{\ast}|^{\frac{\log (\log g) +1}{\log (g-1)}}.\end{split} \end{equation}
Therefore, for sufficiently large $g,$ the above is bounded by 
$o(|[X,X+H]^{\ast}|).$
\end{proof}
\begin{Prop}\label{Prop3.3} 
\begin{equation}\begin{split}
\sum_{n \in [X,X+H]} d_{4}^{\sharp}(n) e(an/q) e(n\beta) &= \int_{X}^{X+H} \sum_{m \leq X^{3\varepsilon}/20} \sum_{b (q)} e(\frac{a}{q}mb) \frac{1}{mq} P_{m} (\log u) e(\beta u) du  \\& +O\left(qX^{2\varepsilon}(1+H|\beta|)\right)\end{split}\end{equation}
\end{Prop}
\begin{proof}
By the definition of $d_{4}^{\sharp}(n),$ 
$$\sum_{n \in [X,X+H]} d_{4}^{\sharp}(n) e(an/q) = \sum_{m \leq X^{3\varepsilon/20}}\sum_{k \in [X/m, (X+H)/m]} P_{m}(\log mk) e(amk/q).$$
The right hand side of the above equation is 
$$\sum_{m \leq X^{3\varepsilon/20}} \sum_{ b(q)} e(amb/q) \sum_{ X/qm \leq k \leq (x+H)/qm} P_{m}(\log m (qk + b))+ O\left(X^{3\varepsilon/20+o(1)}q\right)$$
Removing $b$ inside of $P_{m}$ cause negligible error ,so we have
$$\sum_{m \leq X^{3\varepsilon/20}} \sum_{ b(q)} e(amb/q) \sum_{ X/qm \leq k \leq (X+H)/qm} \left(P_{m}(\log mqk) + O \left(\log X)^{2}d_{3}(m)/k\right)\right).$$
By using the fact that 
$$\sum_{ X/qm \leq k \leq (X+H)/qm} P_{m}(\log mqk) = \frac{1}{mq}\int_{X}^{X+H} P_{m} (\log u) du + O\left(X^{\varepsilon}\right),$$ we have 

\begin{equation}
\sum_{n \in [x,x+H]} d_{4}^{\sharp}(n) e(an/q) =   \sum_{m \leq X^{3\varepsilon/20}} \sum_{ b(q)} e(amb/q)\frac{1}{mq}\int_{X}^{X+H} P_{m} (\log u) du + O(qX^{2\varepsilon}).
\end{equation}
By using integration by parts,
\begin{equation}\begin{split}
\sum_{n \in [x,x+H]} d_{4}^{\sharp}(n) e(an/q) e(n\beta) &=   \sum_{m \leq X^{3\varepsilon/20}} \sum_{ b(q)} e(amb/q)\frac{1}{mq}\int_{X}^{X+H} P_{m} (\log u) e(u\beta) du 
\\& \quad + O\left(qX^{2\varepsilon}(1+ H|\beta|)\right).
\end{split}\end{equation}

\end{proof}

\begin{Prop}\label{Prop1} Let $X^{3/5+\varepsilon} \leq H \leq X^{1-\varepsilon}.$ Then
\begin{equation}\label{majorerror}\begin{split}
&\sum_{n \in [X,X+H]^{\ast}}  \int_{\alpha \in \mathcal{M}} S_{4}^{\sharp}(\alpha;H) e(-n \alpha)  \, d\alpha
\\&\quad = \sum_{q \leq Q} \sum_{a (q)}^{\ast} \sum_{l_{1}|q}\sum_{ b(q/l_{1})}^{\ast} \sum_{l_{2}|q}\sum_{m' \leq X^{3\varepsilon/20}/l_{2}}  \sum_{n \in [X,X+H]^{\ast}}  \frac{1}{m'l_{2}} e(a(m'l_{2}bl_{1}-n)/q) K(H,m',l_{2},n,q) \\& + O\left(X^{5\varepsilon}H^{3\varepsilon}\right)
\end{split}\end{equation}
where 
$$K(H,m',l_{2},n,q):=\int_{X}^{X+H}  \frac{P_{m'l_{2}}(\log u)}{q(u-n)} \sin (2\pi \beta(\delta)(u-n)) du.$$
\end{Prop}
\begin{proof} After integrating over $\alpha$ and by Proposition \ref{Prop3.3}, we only need to consider the error term contributions.
The contribution from the error terms is bounded by 
$$ qX^{2\varepsilon}\left(1+H\beta(\delta)\right)\int_{\alpha \in \mathcal{M}}  |[X,X+H]^{\ast}| F_{[X,X+H]}(\alpha) d \alpha. $$
By using the bounds of the error term in Proposition \ref{Prop3.3} and the exponential sum bound in Lemma \ref{may}, the above term is bounded by 
$$O\left(X^{5\varepsilon}H^{\varepsilon}\right).$$




\end{proof}
Now, we consider the contribution of the minor arcs.
\begin{Prop}\label{Prop3.5}

\[
\sum_{n \in [X,X+H]^{\ast}} \int_{\alpha \in m} S_{4}^{\sharp}(\alpha;H) e(-\alpha n) \, d\alpha =o\left(  |[1.H]^{\ast}| \right).\]

\end{Prop}
\begin{proof}
The contribution from the minor arc is bounded by 
$$ \sup_{\alpha \in m}|S_{4}^{\sharp}(\alpha;H)| \int_{\alpha \in m}  |[X,X+H]^{\ast}| F_{[X,X+H]}(\alpha) d \alpha.  $$
By using the bounds of the supremum in Proposition \ref{Prop4.4}, the proof is completed.
\end{proof}

 By considering the number of multiples of $q \leq Q$ in $[X,X+H]^{\ast},$ we can obtain an upper bound with a few extra power of logarithm.
\begin{Lemma}\label{Mau1} Let $g$ be a sufficiently large natural number and let \[
\int_{[0,1]} \left|\sum_{n \in [1,x]^{\ast}} e(n \alpha)\right| d\alpha \ll x^{\frac{\log \left((\log g)+1\right)}{\log (g-1)}} 
\] for sufficiently large $x.$ Then 
for every \( A > 0 \), there exists \( B>0 \) such that
\begin{equation}\label{Mau} 
\sum_{\substack{q \leq Q (\log X)^{-B} \\ (q, g(g-1)) = 1}} \max_{a \, (\bmod q)} \left|\left| \left\{ n \in [1,X]^{\ast} : n \equiv a \, (\bmod q) \right\}\right| - \frac{|[1,X]^{\ast}|}{q} \right| 
\ll_{g} |[1,X]^{\ast}| (\log X)^{-A}.
\end{equation}
\end{Lemma}
\begin{proof} This is a special case of the results in \cite{DARTYGE2001230}.
\end{proof}

\subsection{Proof of Theorem \ref{Theorem1.1}.}
By Proposition \ref{Prop3.5}, we only need to consider the major arcs contributions. 
Let \( L_q = \beta(\delta) \), and let $P_{m,q}(u):=\frac{P_{m}(\log u)}{q}.$ By separating the integral in Proposition \ref{Prop1}, we see that 
\begin{equation}\begin{split}
\frac{1}{\pi}\int_X^{X+H} \frac{P_{m',q}(u)}{u-n} \sin \left( L_q (u - n) \right) du &\ll \left| \int_{L_q(X - n)}^{-L_{q}} + \int_{L_{q}}^{L_q(u +H - n)} P_{m',q}\left( n + \frac{u}{L_q} \right) \frac{\sin u}{u} du \right| \\& + \left| \int_{-L_q}^{L_q}  P_{m',q}(u)\left( n +\frac{u}{L_q} \right) \frac{\sin u}{u} du \right|.
\end{split}\end{equation}
Using the fact that
\[
\left| \int_0^x \frac{\sin y}{y} dy - \frac{\pi}{2} \right| \ll \frac{1}{|x|}
\]
for large \( x \neq 0 \), 
and 
\[
\frac{\sin x}{x} = 1 + o(x^{2})
\]
for small $x,$ 
we have
$$\int_{0}^{x} \frac{ \sin y}{y}dy \ll \min \left(\frac{\pi}{2}, |x|\right).$$
Also noting that $\left( P_{m',q}(x)\right)' \ll \frac{d_{3}(m')}{qx} .$
By applying integration by parts, we have
\begin{equation}\label{2inte}\left| \int_{L_q(X - n)}^{-L_q}  P_{m',q}(x)\left( n +\frac{x}{L_q} \right) \frac{\sin x}{x} dx \right| \ll \frac{d_{3}(m')}{q} + \left|\int_{L_{q}(X-n)}^{-L_{q}} \frac{d_{3}(m')}{qL_{q}\left(n+x/L_{q}\right)} \min \left(\frac{\pi}{2}, |x|\right) dx\right|
.\end{equation}
 Therefore, the integral of the right-hand side of $\eqref{2inte}$ is bounded by 
$$O\left(d_{3}(m')\left(\frac{H\delta^{-c_{1}}D(\delta)^{-1}}{qX}+\frac{1}{q}\right)\right).$$
Hence, the left-hand side of \eqref{2inte} is bounded by $d_{3}(m')/q.$
By the similar argument as the above,
we have
$$\left| \int_{L_q}^{L_q(X+H-n)}  P_{m',q}\left( n +\frac{x}{L_q} \right) \frac{\sin x}{x} dx \right| \ll \frac{d_{3}(m')}{q},$$ and 
\[
\left| \int_{-L_q}^{L_q} P_{m',q}\left( n +\frac{x}{L_q} \right) \frac{\sin x}{x} dx \right| \ll \frac{d_{3}(m')}{q}.
\]
Thus, the entire contribution from the main term is bounded by
\[
 \sum_{m \leq X^{3\varepsilon/20}} \frac{1}{m}\sum_{q \leq Q} \sum_{n \in [X,X+H]^{\ast}} \frac{d_{2}(q)|c_q(n)|d_{3}(m)}{q}.
\]
Using the crude bound \( |c_q(n)| \leq (n, q) \), we have 

\begin{equation}\label{on}
\begin{split}
&\sum_{q \leq Q} \sum_{n \in [X,X+H]^{\ast}} \frac{|c_q(n)|d_{2}(q)}{q} 
\\&\ll \sum_{a|g(g-1)} d_{2}(a)\sum_{q \leq Q \atop (q, g(g-1)) = a} \frac{d_{2}(q)}{q} \sum_{l|q} l \sum_{n \in [X,X+H]^{\ast}  \atop  (n, q) = l} 1 \\
&\ll \sum_{a|g(g-1)} \frac{d_{2}(a)}{a} \sum_{q \leq Q/a  \atop  (q, g(g-1) = 1} \frac{d_{2}(q)}{q} \sum_{l|aq} l \sum_{n \in [X,X+H]^{\ast}  \atop  l|n} 1 \\
&\ll \sum_{a|g(g-1)} d_{2}(a)\sum_{l \leq aQ} d_{2}(l)\sum_{n \in [X,X+H]^{\ast}  \atop  l|n}  \sum_{lq_1 \leq Q \atop  (lq_1, g(g-1)) = a} \frac{d_{2}(q_{1})}{q_1} \\
&\ll (\log X)^{2} \sum_{a|g(g-1)} d_{2}(a)\sum_{l \leq aQ}d_{2}(l) \sum_{n \in [X,X+H]^{\ast} \atop  l|n} 1.
\end{split}
\end{equation}

By applying \eqref{Mau} and summation by parts, we see that 
\begin{equation}\begin{split}
\sum_{l \leq aQ} d_{2}(l)\sum_{n \in [X,X+H]^{\ast} \atop  l|n} 1 &\ll_{A} |[X,X+H]^{\ast}| \left(\sum_{l \leq aQ} \frac{d_{2}(l)}{l}+ (\log X)^{-A} \right) \\& \ll |[X,X+H]^{\ast}| (\log X)^{2}.
\end{split}\end{equation}
Note that the coprimality condition in Lemma \ref{Mau1} can be easily resolved. Therefore, consider the sum over $m$ 
$$\sum_{m \leq X^{3\varepsilon/20}} \frac{d_{3}(m)}{m} \ll (\log X)^{3},$$ the proof is completed.
\section{Minor arcs}
Now we are left to treat the minor arcs contributions. 
\begin{Lemma}\label{Vinogradov}
Let $\alpha \in [0,1),$ and suppose that $\| a\alpha\|\leq \frac{A}{H} \delta^{-1}$ for at least $D(\delta) A \geq 1$ values of $a \in [A,2A].$ Then there exists a positive integer $q \leq D(\delta)^{-1}$ such that $\| \alpha q\| \leq \frac{q}{H \delta D(\delta)}.$
\end{Lemma}
\begin{proof} We basically follow the argument as in the proof in \cite[Lemma 1.1.14]{Taobook1}. By the pigeonhole principle, there exist two different integers $a, a' \in [A.2A]$ such that $ |a - a'| \leq D(\delta)^{-1}.$ Set $q= a - a',$ so we have 
$ \| q\alpha \| \leq 2 \frac{A}{H}\delta^{-c_{1}}.$ Since $AD(\delta) \gg 1,$ we can consider $q$ arithmetic progressions, and by the pigeonhole principle, there exists a $r \in \{0,1,2,...,q-1\}$ 
such that 
$$\{ A/q + O(1) \leq n \leq 2A/q + O(1) : \| \alpha (nq+ r)\| \leq \frac{A}{H}\delta^{-1}\} $$ has cardinality at least $D(\delta)A/q.$ By considering intervals in the above set, we see that at most one interval (longer than length 1) appears in the set. Therefore, $2 \frac{A}{H} \delta^{-c_{1}} / \|q \alpha\| \geq D(\delta)A/q. $
 
 \end{proof}
Now, let us define $(\delta, A)$ Type I sums, which are somewhat simplified versions of those in \cite{MSTT1}. For the details, see \cite{MSTT1}.
\begin{Def}
Let $0<\delta<1.$ A $(\delta, A)$ type I sum is an arithmetic function of the form $f(n)= \alpha \ast \beta (n):= \sum_{d|n} \alpha(d) \beta(n/d),$ where $\alpha$ is supported on $[1,A],$ and satisfies
$$\sum_{n \leq N} |\alpha(n)|^{2} \leq \frac{1}{\delta} N,$$
$$\| \beta\|_{TV(\mathbb{N},q)} \leq 1$$ for all $N \geq 1$ and for some $1 \leq q \leq 1/\delta .$
\end{Def}
In \cite[Lemma 2.19]{MRSTT2}, the authors proved the following lemma. 
\begin{Lemma}\cite[Lemma 2.19]{MRSTT2}
There exists $J \ll 1$ and sequences $\{a_{i}(n)\}_{i=1}^{J},$ supported on $[1,X^{\varepsilon/5}]$ with $a_{i}(n) \ll d_{3}(n)^{3}$ such that, for each $1 \leq n \leq x$ we have 
$d_{4}^{\sharp}(n) = \sum_{ 1 \leq i \leq J} (a_{j} \ast \phi_{j})(n)$ where $\phi_{j}(n)= (\log n)^{l_{j}} / (\log x)^{l_{j}}$ for some $0 \leq l_{j} \leq 3.$
\end{Lemma}
Therefore, by using Shiu's theorem, it is easy to see that $d_{4}^{\sharp}(n)$ is a linear combinations of $((\log^{1-3^{6}}(X),A)$ Type I sums, where $A \in [1, X^{J\varepsilon/5}].$ 
Now, we are going to use the following inverse theorem.
 \begin{Prop}\label{Prop4.4}
 Let $2 \leq H \leq x.$ Let $f$ be a $((\log X)^{1-3^{6}}, A)$ type I sum for some $A>1,$ such that 
   $$\left|\sum_{x \leq n \leq x+H} f(n)e(\alpha n)\right| \geq \delta H.$$

 Then either $H \ll \delta^{-O(1)}A $ or  
 $$\|q \alpha \| \ll \frac{q}{H\delta D(\delta)}$$ for some $q \ll D(\delta)^{-1}.$
 \end{Prop}
\begin{proof} Assume that $H > \delta^{-O(1)}A.$
We follow the argument in the proof of \cite[Theorem 4.2]{MSTT1}. Since $f= \alpha \ast \beta,$ 
$$\delta H \leq \sum_{a \leq A} |\alpha (a)|\left| \sum_{X/a \leq b \leq X/a +H/a} \beta(b) e( ab \alpha)\right|.$$
By the pigeonhole principle, there exists $1 \leq B \leq A$ such that 
$$\sum_{B \leq a \leq 2B} |\alpha(a)| |\left| \sum_{X/a \leq b \leq X/a +H/a} \beta(b) e( ab \alpha)\right| \geq \delta H / \log x.$$
By applying the Cauchy-Schwartz inequality, 
$$\sum_{B \leq a \leq 2B} \left( \left| \sum_{X/a \leq b \leq X/a +H/a} \beta(b) e( ab \alpha)\right|\right)^{2} \geq \frac{\delta^{2} H^{2}}{ (\log X)^{2}} \frac{1}{B (\log B)^{3^{6}-1}}.$$
Let $D(\delta)B \ll 1.$ Then there exists $a \in [B,2B]$ such that 
$$  \left| \sum_{X/a \leq b \leq X/a +H/a}  e( ab \alpha)\right| \gg \delta H/(B (\log B)^{(3^{6}-1)/2}\log X).$$ Therefore, 
$ \| a \alpha \| \leq \frac{B(\log X)^{(3^{6}+1)/2}}{\delta H}.$
Let $D(\delta)B \gg 1.$ Then
there should be at least $D(\delta)B$ choices of $a \in [B,2B]$ such that 
$$  \left| \sum_{X/a \leq b \leq X/a +H/a}  e( ab \alpha)\right| \gg \delta H/B.$$
Therefore, by applying Lemma \ref{Vinogradov}, the proof is completed. 
\end{proof}
By combining the above results, we get the following:
\begin{Prop}\label{Prop4.5}
Let $\alpha \in m.$ Then
$$S_{4}^{\sharp}(\alpha;H) \ll \delta H.$$ 
\end{Prop}





\bibliographystyle{plain}   
\bibliography{over}  
\end{document}